\documentclass[12pt]{article}
\usepackage[a4paper, left=1.5cm, right=1.5cm, top=2.5cm, bottom=2.5cm, headsep=1.2cm]{geometry}               
\usepackage[parfill]{parskip}   
\usepackage{graphicx}
\usepackage{amssymb}
\usepackage{color}
\usepackage{amsfonts}
\usepackage{amsmath} 
\usepackage{amsthm}
\usepackage{empheq}
\usepackage{esint}
\usepackage{epstopdf}
\usepackage{verbatim}
\usepackage{ifthen}
\usepackage{comment}
\usepackage{hyperref}

\begin{document}

\newcommand{\ov}[1]{\overline{#1}}
\newcommand{\un}[1]{\underline{#1}}
\newcommand{\n}[1]{\lVert#1\rVert}
\newcommand{\eq}[1]{\begin{align}#1\end{align}}
\newcommand{\red}{{\color{red}RED}}
\newcommand{\blue}{{\color{blue}BLUE}}
\newcommand{\green}{{\color{green}GREEN}}
\newcommand{\pink}{{\color{pink}PINK}}
\providecommand{\bs}{\begin{subequations}}
\providecommand{\es}{\end{subequations}}
\newcommand{\com}[1]{}

\newtheorem{theo}{Theorem}
\newtheorem{lem}{Lemma}
\theoremstyle{remark}
\newtheorem{rem}{Remark}

\title{A model of morphogen transport\\ in the presence of glypicans I}

\author{Marcin Ma\l ogrosz
\footnote{Institute of Applied Mathematics and Mechanics, University of Warsaw, Banacha 2, 02-097 Warsaw, Poland (malogrosz@mimuw.edu.pl)}}
\date{}
\maketitle

\begin{abstract}
\noindent
We analyze a one dimensional version of a model of morphogen transport, a biological process governing cell differentiation. The model was proposed by Hufnagel et al. to describe the forming of morphogen gradient in the wing imaginal disc of the fruit fly. In mathematical terms the model is a system of reaction-diffusion equations which consists of two parabolic PDE's and three ODE's. The source of ligands is modelled by a Dirac Delta. Using semigroup approach and $L_1$ techniques we prove that the system is well-posed and possesses a unique steady state. All results are proved without imposing any artificial restrictions on the range of parameters. 
\end{abstract}



\section{Introduction}

Morphogen transport (MT) is a biological process occurring in the tissues of living organisms. It is known that certain proteins (ligands) act as the morphogen - a conceptually defined substance which is responsible for the development and differentiation of cells. As it is proposed in the 'French flag model' by Wolpert \cite{Wol}, morphogen molecules spread from a spatially localized source through the tissue and after some time form stable concentration gradient. According to the concept of positional signalling, receptors located on the surface of the cells, detect information about local levels of morphogen concentration. This information is transmitted to nucleus and cause gene activation which finally leads to the synthesis of suitable proteins and cell differentiation. Although the role of morphogen gradient in gene expression seems to be widely accepted, the exact kinetic mechanism of its formation is still not known (see \cite{GB},\cite{KPBKBJG-G},\cite{KW}).

Recently various models of MT, consisting of PDE-ODE systems, were proposed (see \cite{LNW},\cite{BKPG-GJ},\cite{Huf}) and analyzed (see \cite{KLW1},\cite{KLW2},\cite{Tel},\cite{Mal},\cite{STW}). These models assume that movement of morphogen molecules occur by different types of diffusion or by chemotaxis in the extracellular medium. Reactions with receptors (reversible binding, transcytosis) and various possibilities of degradation and internalization (of morphogens, receptors, morphogen-receptor complexes) are also being considered. 

In \cite{Huf} Hufnagel et al. proposed a model, which we denote \textbf{[HKCS]}, of the formation of morphogen Wingless (Wg) in the Drosophila Melanogaster wing imaginal disc. Apart from mentioned before processes, the model takes into account interaction of Wg with glypican Dally-like (Dlp) - protein which, similarly to receptor, interacts with morphogen through association-dissociation mechanism. Dlp molecules may also transmit Wg to each other causing the movement of morphogen particles on the surface of the wing disc.  
The interesting issue is the presence of a singular term (a Dirac Delta), to model the secretion of morphogen from a narrow part of the tissue which is represented in the model by a point source.
In the present paper we analyze a one dimensional simplification of the \textbf{[HKCS]} model, which was also introduced in \cite{Huf}. We intend to extend our analysis to the case of higher dimensional domains in the forthcoming article(s).

\subsection{The \textbf{[HKCS].1D} model.}
In this section we present a one dimensional simplification of the model \textbf{[HKCS]} introduced in $\cite{Huf}$.\\
\newline
For $L>0, \ \infty\geq T>0$, denote $I^L=(-L,L), \ \partial I^L=\{-L,L\}, \ I^L_T=(0,T)\times I^L, \ (\partial I^L)_T=(0,T)\times\partial I^L, \ I=I^1$.
The following system is of our interest:\\
\newline
\textbf{[HKCS].1D}
\bs\label{Model}
\eq{
\partial_tW&=D\partial^2_{xx}W-\gamma W - [kGW-k'W^*] - [k_RRW-k_R'R^*] + s\delta, &&(t,x)\in I^L_{\infty}\\
\partial_tW^*&=D^*\partial^2_{xx}W^*-\gamma^* W^*+[kGW-k'W^*]-[k_{Rg}RW^*-k_{Rg}'R_g^*], &&(t,x)\in I^L_{\infty}\\
\partial_tR&= -[k_RRW-k_R'R^*]-[k_{Rg}RW^*-k_{Rg}'R_g^*]-\alpha R+\Gamma, &&(t,x)\in I^L_{\infty}\\
\partial_tR^*&=[k_RRW-k_R'R^*] -\alpha^*R^*, &&(t,x)\in I^L_{\infty}\\
\partial_tR_{g}^*&=[k_{Rg}RW^*-k_{Rg}'R_g^*]-\alpha^*R_{g}^*, &&(t,x)\in I^L_{\infty}\\
\partial_xW&=\partial_xW^*=0, &&(t,x)\in(\partial I^L)_{\infty}\\
W(0)&=W_0, \ W^*(0)=W^*_0,\  R^*(0)=R^*_0,\ R_g^*(0)=R_{g0}^*,\ R(0)=R_0, &&x\in I^L
}
\es

In \eqref{Model} $W, G, R$ denote concentrations of free morphogens Wg, 
free glypicans Dlp and free receptors,\\ $W^*, R^*$ denote concentrations of morphogen-glypican
and morphogen-receptor complexes,\\ $R_g^*$ denotes concentration of morphogen-glypican-receptor complexes.\\
The model takes into account association-dissociation mechanism of
\begin{itemize}
\item $W$ and $G$ with rates $k, k'$ ($kGW-k'W^*$),
\item $W$ and $R$ with rates $k_R, k_R'$ ($k_RRW-k_R'R^*$),
\item $W^*$ and $R$ with rates $k_{Rg}, k_{Rg}'$ ($k_{Rg}RW^*-k_{Rg}'R_g^*$).
\end{itemize}

Other terms of the system account for
\begin{itemize}
\item linear diffusion of $W$ $(W^*)$ with rate $D$ $(D^*)$: $-D\partial^2_{xx}W$ $(-D^*\partial^2_{xx}W^*)$,
\item degradation of $W$ $(W^*)$ with rate $\gamma$ $(\gamma^*)$: $-\gamma W$ $(-\gamma^* W^*)$,
\item internalization (endocytosis) of $R$ $(R^*,R_g^*)$ with rate $\alpha$ $(\alpha^*)$ : $-\alpha R$ $(-\alpha^* R^*, -\alpha^* R_g^*)$,
\item secretion of $W$ with rate $s$ from the source localised at $x=0\in I$: $s\delta$ ($\delta$ denotes the Dirac Delta),
\item production of receptors: $\Gamma$.
\end{itemize}

For simplicity we assume that $G$ and $\Gamma$ are given, strictly positive constants.

\subsection{Nondimensionalisation}

To analyze system \textbf{[HKCS].1D} we rewrite it in a nondimensional form:
\bs\label{System}
\eq{
\partial_t u_1-\partial^2_{xx} u_1&=-(b_1+c_1+u_3)u_1+c_2u_2+c_4u_4+p_1\delta,&&(t,x)\in I_{\infty}\label{SystemA}\\
\partial_t u_2-d\partial^2_{xx} u_2&=-(b_2+c_2+c_{3}u_3)u_2+c_1u_1+c_5u_5,&&(t,x)\in I_{\infty}\label{SystemB}\\
\partial_t u_3&=-(b_3+u_1+c_{3}u_2)u_3+c_4u_4+c_5u_5+p_3,&&(t,x)\in I_{\infty}\label{SystemC}\\
\partial_t u_4&=-(b_4+c_4)u_4+u_1u_3,&&(t,x)\in I_{\infty}\label{SystemD}\\
\partial_t u_5&=-(b_5+c_5)u_5+c_{3}u_2u_3,&&(t,x)\in I_{\infty}\label{SystemE}
}
\es
with boundary and initial conditions
\begin{align*}
\partial_x u_1&=\partial_x u_2=0,&&(t,x)\in(\partial I)_{\infty}\\
\mathbf{u}(0,\cdot)&=\bold{u_0},&&x\in I
\end{align*}
where
\begin{align*}
\mathbf{u}(t,x)&=(u_1,u_2,u_3,u_4,u_5)(t,x)=K(W,W^*,R,R^*,R_g^*)(Tt,Lx),\\
\bold{u_0}(x)&=(u_{10},u_{20},u_{30},u_{40},u_{50})(x)=K(W_0,W^*_0,R_0,R^*_0,R_{g0}^*)(Lx),\\
T&=L^2/D,\ K=Tk_R ,\ d=D^*/D,\\
\bold{b}&=(T\gamma,T\gamma^*,T\alpha,T\alpha^*,T\alpha^*),\\
\bold{c}&=(TkG,Tk',k_{Rg}/k_R,Tk_R',Tk_{Rg}'),\\
\bold{p}&=(KTs,0,KT\Gamma,0,0).
\end{align*}

\subsection{Overview.}
The aim of this paper is to establish well-posedness of \eqref{System} and the existence of a unique steady state.

During the analysis of the system \eqref{System}
we encounter the following difficulties:
\begin{itemize}
\item absence of diffusion in equations, \eqref{SystemC},\eqref{SystemD}, \eqref{SystemE} so that there is no smoothing effect for  $u_3,u_4,u_5$.
\item singular source term in \eqref{SystemA},
\item nonsymmetric zero order part of the operator for the stationary problem.
\end{itemize}

We first solve the stationary problem for $\eqref{System}$ by using Schauder's fixed point theorem. The key observation is that the linear operator which appears in the definition of $T_n(v)$ (see proof of Theorem 1), has a diagonally dominant structure. This leads us to analyze the problem in an $L_1(I)$ setting rather than $L_2(I)$. To prove uniqueness we consider the system which is satisfied by the difference of two possible solutions and after algebraic manipulations show that it also has a diagonally dominant structure.\\
To remove the singularity $p_1\delta$ from \eqref{SystemA} we change variables $\bold{z}=\bold{u}-\bold{u^*}$, where $\bold{u^*}$ is the steady state to \eqref{System}. Then local well-posedness in the space of continuous functions of the system for $\bold{z}$  follows from the classical perturbation theory for sectorial operators. To prove global existence we notice that the quasipositivity of the vector field appearing on the right hand side of $\eqref{System}$ guarantees that the semiflow generated by $\eqref{System}$ preserves the positive cone. Then using compensation effects it is easy to show that $u_3,u_4,u_5\in L_{\infty}(0,T_{max};C(I))$ and $u_1,u_2\in L_{\infty}(0,T_{max};L_1(I))$. Finally using smoothing effects of the heat semigroup we prove that $u_1,u_2\in L_{\infty}(0,T_{max};C(I))$, from which we conclude that system $\eqref{System}$ is globally well-posed and has bounded trajectories.   

Before stating the results precisely we introduce the notation and function spaces in which we will analyze the system \eqref{System}.
\section{Notation and function spaces}
For $x,y\in\mathbb{R},\bold{x},\bold{y}\in\mathbb{R}^n$ we denote 
\begin{align*}
x\vee y&=\max\{x,y\}, \ x\wedge y=\min\{x,y\},\
x_{+}=x\vee 0, \ x_{-}=(-x)\vee 0,\
sgn(x)=\begin{cases}|x|/x &, x\neq0\\ 0 &, x=0\end{cases},\\
\overline{\bold{x}}&=\max\{x_i\colon 1\leq i\leq n\},\ \underline{\bold{x}}=\min\{x_i\colon 1\leq i\leq n\},\ 
<\bold{x},\bold{y}>=\sum_{i=1}^nx_iy_i.
\end{align*} 
If $V$ is a vector space we denote by $V^n$ the n-th product power of $V$. If $(V,\geq)$ is a partially ordered vector space we denote its positive cone by $V_{+}:=\{v\in V\colon \ v\geq 0\}$. If $V,W$ are normed spaces we denote by $\mathcal{L}(V,W)$ the space of linear, bounded operators from $V$ to $W$ with the usual uniform convergence topology. We put $\mathcal{L}(V)=\mathcal{L}(V,V)$. 

We make standard convention that $C$ denotes a positive constant which depends on various parameters which are specified explicitly in the text.

To analyze the problem we will use the following Banach spaces
\begin{align*}
X&=X_0=C(\overline{I}), \ X_1=C^2_N(\overline{I})=\{u\colon u\in C^2(\overline{I}), \ u'(-1)=u'(1)=0\}, \\
X_{1/2}&=C^{0,1}(I)=Lip(I)=W^1_{\infty}(I),\\
Y&=Y_0=L_1(I), \ Y_1=W^2_{1,N}(I)=\{u\colon u\in W^2_1(I), u'(-1)=u'(1)=0\},\\
Y_{1/2}&=W^1_1(I)=AC(I).
\end{align*}

Notice that due to the imbedding $W^2_1(I)\subset C^1(\overline{I})$ the boundary conditions in the definition of $Y_1$ are meaningful.

\section{Results}
From now on we assume that
\begin{align*}
d,\bold{b}>0, \ 
\bold{c},\bold{p}\geq0, \
\bold{u_0}\in X^5_{+}.
\end{align*}
We start with the analysis of the stationary problem and prove that there exists unique nonnegative steady state. Observe that due to the absence of diffusion in \eqref{SystemC},\eqref{SystemD},\eqref{SystemE} the stationary problem reduces to the system \eqref{SS2} (see below) of two semilinear elliptic equations for $u_1^*$ and $u_2^*$.\\
\begin{theo}
System $\eqref{System}$ possesses a unique nonnegative steady state\\
$\bold{u^*}\in X_{1/2}\times X_1\times X_{1/2}^3$
such that
\bs\label{miau}
\eq{
u^*_3=H(u^*_1,u^*_2), \ b_4u^*_4=k_1u_1^*H(u^*_1,u^*_2), \ b_5u^*_5=k_2u_2^*H(u^*_1,u^*_2),
}
\es
where
\begin{align}
k_1=b_4/(b_4+c_4), \ k_2=c_{3}b_5/(b_5+c_5), \ H(x_1,x_2)=p_3/(k_1x_1+k_2x_2+b_3)\label{Hdef}
\end{align}
and $(u^*_1,u^*_2)$ is a solution of the following boundary value problem
\bs\label{SS1}
\eq{
-{u^*_1}''+(b_1+c_1+k_1H(u^*_1,u^*_2))u^*_1-c_2u^*_2&=p_1\delta,&&x\in I\label{SS1a}\\
-d{u^*_2}''-c_1u^*_1+(b_2+c_2+k_2H(u^*_1,u^*_2))u^*_2&=0,&&x\in I\label{SS1b},\\
{u^*_1}'&={u^*_2}'=0,&&x\in\partial I.
}
\es
i.e. for every $\varphi\in X_{1/2}$
\begin{align*}
\int_I[{u^*_1}'\varphi'+((b_1+c_1+k_1H(u^*_1,u^*_2))u^*_1-c_2u^*_2)\varphi]=s_1\varphi(0)
\end{align*}
and \eqref{SS1b} is satisfied in the classical sense.
Moreover
\begin{align}\label{higherreg}
u_1^*+p_1|x|/2\in X_1.
\end{align}

\end{theo}

A typical shape of the steady state is to be found in Figure 1. 

\begin{figure}[h]
\includegraphics[width=\textwidth]{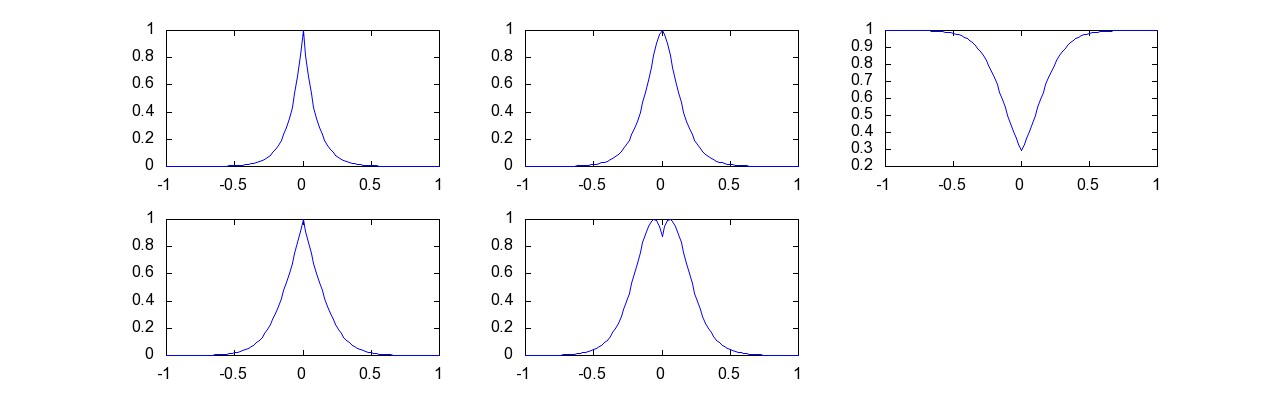}
\caption{Graph of $\bold{u}^*$ (normalised to 1) computed for the following values of parameters: $\bold{b}=[100,10,10,10,10],\ \bold{c}=[10,10,1,10,10],\ \bold{p}=[100,0,100,0,0],\ d=1/10$.
First row - $u_1^*/\n{u_1^*}_X$,$u_2^*/\n{u_2^*}_X$,$u_3^*/\n{u_3^*}_X$, second row - $u_4^*/\n{u_4^*}_X$,$u_5^*/\n{u_5^*}_X$. Notice the surprising difference in behavior of $u_4^*$ and $u_5^*$ near $x=0$ (see Remark \eqref{rem1} for explanation).}
\label{wykres}
\end{figure}

In the following remark we analyze the behavior of the stationary solution near the source of morphogen.\\
\begin{rem}[$\bold{u^*}$ near $x=0$]\label{rem1}
Observe that as $\bold{u^*}$ is unique it must be even. Indeed otherwise $\bold{u^*}(-x)$ would be a second solution as the system \eqref{SS1} is invariant under the transformation $x\to-x$. Thus using \eqref{higherreg}
\bs\label{ble}
\eq{
(u_1^*)'(0^+)&=-p_1/2<0\label{ble1},\\
(u_2^*)'(0)&=0,
}
\es
where $(u_1^*)'(0^+)$ denotes the right-sided derivative of $u_1^*$ at $x=0$. 

Using \eqref{miau} and \eqref{ble} we compute directly
\bs\label{bla}
\eq{
(u_3^*)'(0^+)&=\frac{p_1k_1}{2p_3}[H(u_1^*(0),u_2^*(0))]^2>0,\label{bla1}\\
(u_4^*)'(0^+)&=-\frac{p_1k_1k_2}{2p_3b_4}[H(u_1^*(0),u_2^*(0))]^2(u_2^*(0)+b_3/k_2)<0,\label{bla2}\\
(u_5^*)'(0^+)&=\frac{p_1k_1k_2}{2p_3b_5}[H(u_1^*(0),u_2^*(0))]^2u_2^*(0)>0.\label{bla3}
}
\es
In particular from \eqref{bla2}, \eqref{bla3} and the fact that $\bold{u^*}$ is even we infer that $u_4(0) \ (u_5(0))$ is a strict local maximum (minimum), which explains the difference near $x=0$ in $u_4^*$ (steep spike) and $u_5^*$ (depletion effect) as observed in Figure \eqref{wykres}.\\
\end{rem}

Next we turn our attention to the evolution problem and establish its well-posedness and the uniform boundedness of trajectories in $X^5$.\\ 
\begin{theo}
System $\eqref{System}$ possesses a unique, global in time, nonnegative solution
\bs\label{global_theo}
\eq{
u_1&\in C([0,\infty);X)\cap C^1((0,\infty);X)\cap C((0,\infty);X_{1/2})\\
u_2&\in C([0,\infty);X)\cap C^1((0,\infty);X)\cap C((0,\infty);X_1)\\
u_3,u_4,u_5&\in C^1([0,\infty);X)
}
\es
such that for every $\varphi\in X_{1/2}, t\in(0,\infty)$
\begin{align*}
\int_I\partial_t u_1\varphi+D\int_I\partial_xu_1\partial_x\varphi=\int_I[-(b_1+c_1+u_3)u_1+c_2u_2+c_4u_4]\varphi +p_1\varphi(0)
\end{align*}
and other equations are satisfied in the sense of $X$. Moreover $\bold{u}\in L_{\infty}(0,\infty;X^5)$
and the following estimates hold
\bs\label{EstCalc}
\eq{
\sum_{i=3}^5u_i(t)&\leq e^{-\underline{\bold{b}}t}\sum_{i=3}^5u_{i0}+p_3(1-e^{-\underline{\bold{b}}t})/\underline{\bold{b}}\label{EstCalc1},\\
\sum_{i\in\{1,2,4,5\}}\n{u_i(t)}_Y&\leq e^{-\underline{\bold{b}}t}\sum_{i\in\{1,2,4,5\}}\n{u_{i0}}_Y+p_1(1-e^{-\underline{\bold{b}}t})/\underline{\bold{b}}\label{EstCalc2}.
}
\es
\\
\end{theo}

We conclude with a remark concerning the discussion about the asymptotic behavior. \\

\begin{rem}[Asymptotics]\label{rem2}
For the case of morphogen Dpp acting in the imaginal wing disc of the fruit fly without the presence of glypicans, it is proved in \cite{KLW1} that the morphogen gradient (i.e. steady state of the appropriate evolution system) is globally exponentially stable.  
It is expected that an analogous result should hold for \textbf{[HKCS].1D}, though we are not able to prove even the local stability of the steady state. However it may also be the case that the presence of glypicans within a certain range of parameters has a destabilising effect on the equilibrium.
\end{rem}

\section{Lemmas}

In this section we collect lemmas which are used in the proofs of the results. The first lemma concerns well-posedness of the abstract ODE's in Banach spaces, while the second states that realisations of one dimensional Laplace operator in the chosen Banach spaces are sectorial. Since these lemmas are well known we state them only to make the article more self-contained. For the proofs we refer the interested reader to \cite{Lor}, chapter 6.1 and \cite{Lun}, chapter 3.1. 
\\

\begin{lem}\label{lem3}
Assume that \begin{enumerate}
\item $V_1\subset V$ are arbitrary, densely imbedded Banach spaces.
\item $A:V\supset V_1\to V$ is a sectorial operator.
\item $F:V\to V$ is Lipschitz on bounded subsets.
\end{enumerate}
Then for given $v_0\in V,$ the Cauchy problem
\bs\label{Cauchy}
\eq{
v'-Av&=F(v), &&t>0\\
v(0)&=v_0
}
\es
has a unique, strong, maximal solution
\begin{align*}
v\in C([0,T_{max});V)\cap C^1((0,T_{max});V)\cap C((0,T_{max});V_1).
\end{align*}
The following Duhamel formula holds
\begin{align*}
v(t)=e^{At}v_0+\int_0^te^{A(t-s)}F(v(s))ds, \quad t\in[0,T_{max}).
\end{align*}
Moreover $T_{max}$ is characterised by the following blow-up condition:
\begin{align}
T_{max}<\infty \ {\rm implies} \ \limsup_{t\to T_{max}^-}\n{v(t)}_V=\infty\label{blow}.
\end{align}
In particular if there exists $C>0$ such that for $t\in[0,T_{max})$
\begin{align}
\n{F(v(t))}_V\leq C(\n{v(t)}_V+1)\label{blowcond},
\end{align}
then $T_{max}=\infty$.
\end{lem}

For $Z\in\{X,Y\}$ we define the $Z$-realisation of the  Laplace operator with Neumann boundary conditions on $I$:
\begin{align*}
A_Z&:Z\supset Z_1\to Z, \ A_Zu=u'', \ u\in Z_1
\end{align*}

\begin{lem}\label{lem1}
$A_Z$ is a sectorial, densely defined operator with compact resolvent. It generates an analytic, strongly continuous semigroup $e^{tA_Z}$ and for $t>0$ the following estimates hold
\begin{align*}
\n{e^{tA_Z}}_{\mathcal{L}(Z)}&\leq 1,\quad
\n{e^{tA_Y}}_{\mathcal{L}(Y,X)}\leq C(1\wedge t)^{-1/2}.
\end{align*}
Moreover $(A_X,e^{tA_X})$ is a restriction of $(A_Y,e^{tA_Y})$ to $X$ i.e.
\begin{align*}
A_Xu&=A_Yu, \ u\in X_1,\quad  e^{tA_X}u=e^{tA_Y}u, \ (t,u)\in[0,\infty)\times X.
\end{align*}
\end{lem}

The third lemma concerns solvability of linear elliptic systems with diagonally dominant zero order term. It is crucial in the proofs of existence and uniqueness of the steady state of the system \eqref{System}.\\

\begin{lem}\label{lem2}
Assume that for $i,j=1,2, \ d_i>0,a_{ij}\in X_{+}$ and
\bs\label{aijcond}
\eq{
a_{11}-a_{21}&\geq0,\\
a_{22}-a_{12}&\geq0.
}
\es
Define operators
\begin{align*}
M&:Y^2\to Y^2, \ 
Mu=(-a_{11}u_1+a_{12}u_2,a_{21}u_1-a_{22}u_2),\\
G&:Y^2\supset Y_1^2\to Y^2, \ 
G=(d_1A_Y)\times(d_2A_Y)+M.
\end{align*}
Then $G$ is a sectorial, densely defined operator with a compact resolvent $R(\lambda,G)=(\lambda-G)^{-1}$ and the following hold 
\bs\label{lemthesis}
\eq{
&(0,\infty)\subset\rho(G) \ {\rm and} \  \n{R(\lambda,G)}_{\mathcal{L}(Y^2)}\leq1/\lambda,\label{lemthesis1}\\ 
&\n{R(\lambda,G)}_{\mathcal{L}(Y^2,Y_1^2)}\leq C(1+1/\lambda)\label{lemthesis2},\\
&R(\lambda,G) \ {\rm preserves} \ Y^2_{+}\label{lemthesis3},
}
\es
where $\lambda>0$ and $C$ depends only on $d_i,\n{a_{ij}}_X$.
\end{lem}

\begin{proof}
To prove that $G$ is sectorial and has a compact resolvent notice that it is a perturbation of the operator $(d_1A_Y)\times(d_2A_Y)$ having these two properties by a bounded operator $M\in\mathcal{L}(Y^2)$. From the compactness of the resolvent of $G$ we get that the spectrum $\sigma(G)$ only contains eigenvalues (Theorem 6.29 \cite{Kat}). 
\\In the rest of the proof we will use the following observation. Let 
$\gamma:\mathbb{R}\to\mathbb{R}$ be a function such that $x\gamma(x)\geq0,\  |\gamma|\leq1$
. Then using \eqref{aijcond} we obtain the following pointwise inequality
\bs\label{LaurL}
\eq{
\Big<Mu,(\gamma(u_1),\gamma(u_2))\Big>&=-a_{11}u_1\gamma(u_1)+a_{12}u_2\gamma(u_1)+a_{21}u_1\gamma(u_2)-a_{22}u_2\gamma(u_2)\\
&=-u_1\gamma(u_1)(a_{11}-a_{21}\gamma(u_2)\gamma(u_1))-u_2\gamma(u_2)(a_{22}-a_{12}\gamma(u_1)\gamma(u_2))\leq0.
}
\es
Choose $\lambda>0, \ f\in Y^2, u\in Y_1^2$ such that 
\begin{align}\label{lem3eq1}
f=(\lambda-G)u.
\end{align}
To proove \eqref{lemthesis1} we estimate 
\begin{align*}
\n{f}_{Y^2}&\geq\int_I\Big<f,(sgn(u_1),sgn(u_2))\Big>=\lambda\n{u}_{Y^2}-\sum_{i=1}^2d_i\int_Iu_i''sgn(u_i)\\
&-\int_I\Big<Mu,(sgn(u_1),sgn(u_2))\Big>\geq\lambda\n{u}_{Y^2},
\end{align*}
where we used \eqref{LaurL} with $\gamma=sgn$ and the following Kato's inequality (see Lemma 2 in \cite{BreStr})
\begin{align}
-\int_Iv''sgn(v)\geq0, \quad v\in Y_1\label{KatoIneq}.
\end{align}

To prove \eqref{lemthesis2} observe that from \eqref{lem3eq1},\eqref{lemthesis1} we have
\begin{align*}
\n{Gu}_{Y^2}\leq\n{f}_{Y^2}+\lambda\n{u}_{Y^2}\leq2\n{f}_{Y^2},
\end{align*}
whence
\begin{align*}
\n{u}_{Y_1^2}&\leq C(\n{(A_Yu_1,A_Yu_2)}_{Y^2}+\n{u}_{Y^2})\leq C[(d_1\wedge d_2)^{-1}\n{Gu-Mu}_{Y^2}+\n{f}_{Y^2}/\lambda]\\
&\leq
C[(d_1\wedge d_2)^{-1}(\n{Gu}_{Y^2}+\n{M}_{\mathcal{L}(Y^2)}\n{u}_{Y^2})+\n{f}_{Y^2}/\lambda]\\
&\leq C[(d_1\wedge d_2)^{-1}(2+\n{M}_{\mathcal{L}(Y^2)}/\lambda)+1/\lambda]\n{f}_{Y^2}\\
&\leq C(1+1/\lambda)\n{f}_{Y^2},
\end{align*}

Finally to prove \eqref{lemthesis3} assume that $f\in Y^2_+$. Let $\phi(x)=(sgn(x)-1)/2$. Then $-1\leq\phi(x)\leq0$ and $x\phi(x)=x_-$. Using \eqref{LaurL} with $\gamma=\phi$ and \eqref{KatoIneq} we obtain
\begin{align*}
0\geq\int_I\Big<f,(\phi(u_1),\phi(u_2))\Big>&=\lambda\sum_{i=1}^2\int_Iu_i\phi(u_i)-\frac{1}{2}\sum_{i=1}^2d_i\int_Iu_i''sgn(u_i)-\int_I\Big<Mu,(\phi(u_1),\phi(u_2))\Big>\\
&\geq \lambda\n{u_-}_{Y^2},
\end{align*}
whence $u\geq0$.
\end{proof}

\section{Proof of Theorem 1}
We divide the proof into two parts. To prove existence of a solution of \eqref{SS1} we first approximate the singular source term $p_1\delta$ by more regular functions $h_n\in Y_+$. Using Schauder's fixed point theorem we prove solvability of the approximated problem. Finally using compactness methods we show that the approximated solutions converge to a solution of \eqref{SS1}. In the proof of uniqueness we show that the difference of any two possible steady states belongs to the kernel of a certain operator $\lambda-G$, where $\lambda>0$ and $G$ satisfies assumptions of Lemma \eqref{lem2}.

\subsection{Existence}
Choose a sequence $h_n\in Y_{+}$ such that $h_n\rightharpoonup^*\delta$ in $\mathcal{M}([-1,1])$ - the space of signed Radon measures.  
For $\bold{v}\in X_{+}^2$ consider the following problem
\bs\label{SS2}
\eq{
-u_1''+(b_1+c_1+k_1H(v_1,v_2))u_1-c_2u_2&=p_1h_n,&&x\in I\\
-du_2''-c_1u_1+(b_2+c_2+k_2H(v_1,v_2))u_2&=0,&&x\in I,\\
u_1'=u_2'&=0,&&x\in\partial I,
}
\es
where $H$ is defined in \eqref{Hdef}.
Using notation introduced in Lemma \eqref{lem2} system \eqref{SS2} is equivalent to
\begin{align*}
(\lambda-G)(u_1,u_2)=(p_1h_n,0),
\end{align*}
where
\begin{align*}
\lambda&=\underline{\bold{b}},\\
d_1&=1, d_2=d,\\
a_{11}&=b_1-\underline{\bold{b}}+c_1+k_1H(v_1,v_2),&&a_{12}=c_2,\\
a_{21}&=c_1,&&a_{22}=b_2-\underline{\bold{b}}+c_2+k_2H(v_1,v_2).
\end{align*}
Oserve that condition \eqref{aijcond} holds, thus using Lemma \eqref{lem2} we obtain that \eqref{SS2} has a unique solution $(u_1,u_2)\in Y_{1,+}^2$ and there exists  $C_1$ which do not depend on $(v_1,v_2),(u_1,u_2),h_n$ such that 
\begin{align}
\n{(u_1,u_2)}_{Y_1^2}\leq C_1\n{h_n}_Y\label{est}.
\end{align}
Using the compact imbedding
\begin{align}
Y_1\subset\subset X\label{comp}
\end{align}
and \eqref{est} we obtain that there exists $C_2$ such that
\begin{align}
\n{(u_1,u_2)}_{X^2}\leq C_2\n{(u_1,u_2)}_{Y_1^2}\leq C_1C_2\n{h_n}_Y\label{est2}.
\end{align}
Define 
\begin{align*}
V_n&=\{(v_1,v_2)\in X_{+}^2\colon \n{(v_1,v_2)}_{X^2}\leq C_1C_2\n{h_n}_Y\},\\
T_n&\colon V_n\to V_n, \ T_n(v_1,v_2)=(u_1,u_2), 
\end{align*}
where $(u_1,u_2)$ is the solution of \eqref{SS2}. Observe that $V_n$ is a closed and convex subset of a Banach space $X^2$ and $T_n$ is well defined, compact (because of \eqref{comp},\eqref{est2}) and continuous (because $H$ is a globally Lipschitz continuous function on $\mathbb{R}_{+}^2$). Hence by Schauder's theorem $T_n$ has a fixed point $(u_{n,1}^*,u_{n,2}^*)\in V_n$.\\
Since $(h_n)_{n=1}^{\infty}$ is bounded in $Y$ we get, by \eqref{est2}, that $(u_{n,1}^*,u_{n,2}^*)_{n=1}^{\infty}$ is bounded in $Y_1^2$. From the imbeddings $Y_1\subset X_{1/2}\subset\subset X$ there exist $(u_1^*,u_2^*)\in X_{1/2}$ and a subsequence $(u_{n_k,1}^*,u_{n_k,2}^*)_{k=1}^{\infty}$ such that for $i=1,2$
\begin{align*}
(u_{n_k,i}^*)'&\rightharpoonup^*(u_i^*)', \ \rm{in} \ L_{\infty}(I)\\
u_{n_k,i}^*&\to u_i^*, \ \rm{in} \ X.
\end{align*}
Fix $\varphi\in X_{1/2}$, then since $T_{n_k}(u_{n_k,1}^*,u_{n_k,2}^*)=(u_{n_k,1}^*,u_{n_k,2}^*)$ we have:
\bs\label{SSweak}
\eq{
\int_I(u_{n_k,1}^*)'\varphi'+[(b_1+c_1+k_1H(u_{n_k,1}^*,u_{n_k,2}^*))u_{n_k,1}^*-c_2u_{n_k,2}^*]\varphi&=s_1\int_Ih_{n_k}\varphi,\\
d\int_I(u_{n_k,2}^*)'\varphi'+[-c_1u_{n_k,1}^*+(b_2+c_2+k_2H(u_{n_k,1}^*,u_{n_k,2}^*))u_{n_k,2}^*]\varphi&=0,
}
\es
Using again the fact that $H$ is globally Lipschitz continuous on $\mathbb{R}_+^2$ we can pass in \eqref{SSweak} with $n_k\to\infty$ and obtain that $(u_1^*,u_2^*)$ is a solution of \eqref{SS2}. 

\subsection{Uniqueness}

Assume that $(u_1,u_2),(v_1,v_2)$ are two solutions of \eqref{SS1}. Noting $z_i=u_i-v_i$ for $i=1,2$ we have:
\begin{align*}
-z_1''+(b_1+c_1)z_1-c_2z_2+k_1(H(u_1,u_2)u_1-H(v_1,v_2)v_1)&=0\\
-dz_2''-c_1z_1+(b_2+c_2)z_2+k_2(H(u_1,u_2)u_2-H(v_1,v_2)v_2)&=0
\end{align*}
Define
\begin{align*}
D&=(k_1u_1+k_2u_2+b_3)(k_1v_1+k_2v_2+b_3)\\
w_i&=(u_i+v_i)/2, \ i=1,2
\end{align*}
and compute
\begin{align*}
u_1v_2-u_2v_1&=z_1(u_2+v_2)/2-z_2(u_1+v_1)/2=z_1w_2-z_2w_1\\
H(u_1,u_2)u_1-H(v_1,v_2)v_1&=p_3\Big(\frac{u_1}{k_1u_1+k_2u_2+b_3}-\frac{v_1}{k_1v_1+k_2v_2+b_3}\Big)=
\frac{p_3}{D}(k_2(u_1v_2-u_2v_1)+b_3z_1)\\
&=\frac{p_3}{D}((k_2w_2+b_3)z_1-k_2w_1z_2)\\
H(u_1,u_2)u_2-H(v_1,v_2)v_2&=p_3\Big(\frac{u_2}{k_1u_1+k_2u_2+b_3}-\frac{v_2}{k_1v_1+k_2v_2+b_3}\Big)=
\frac{p_3}{D}(-k_1(u_1v_2-u_2v_1)+b_3z_2)\\
&=\frac{p_3}{D}(-k_1w_2z_1+(k_1w_1+b_3)z_2).
\end{align*}
Thus
\begin{align*}
-z_1''+(b_1+\frac{k_1p_3b_3}{D}+c_1+\frac{k_1k_2p_3w_2}{D})z_1-(c_2+\frac{k_1k_2p_3w_1}{D})z_2&=0\\
-dz_2''-(c_1+\frac{k_1k_2p_3w_2}{D})z_1+(b_2+\frac{k_2p_3b_3}{D}+c_2+\frac{k_1k_2p_3w_1}{D})z_2&=0
\end{align*}

Hence, using the notation introduced in Lemma \eqref{lem2}, $(z_1,z_2)$ belongs to the kernel of the operator $\underline{\bold{b}}-G$ where 
\begin{align*}
d_1&=1,d_2=d\\ 
a_{11}&=b_1-\underline{\bold{b}}+\frac{k_1p_3b_3}{D}+c_1+\frac{k_1k_2p_3w_2}{D}, &&a_{12}=c_2+\frac{k_1k_2p_3w_1}{D},\\ a_{21}&=c_1+\frac{k_1k_2p_3w_2}{D}, &&a_{22}=b_2-\underline{\bold{b}}+\frac{k_2p_3b_3}{D}+c_2+\frac{k_1k_2p_3w_1}{D}.
\end{align*}
Since nonnegativity of $w_1, w_2$ ensures that assumption \eqref{aijcond} is fulfilled we infer that $z_1=z_2=0$ which finishes the proof.

\subsection{Regularity of $u_1^*$ outside $x=0$.}

Observe that $E=-p_1|x|/2$ satisfies $-E''=p_1\delta$
in the sense of distributions. Owing to \eqref{SS1a} $v=u_1^*-E$ solves
the following boundary value problem
\begin{align*}
-v''&=f,&&x\in I\\
v'&=-E',&&x\in \partial I 
\end{align*}
with $f=c_2u_2^*-(b_1+c_1+k_1H(u_1^*,u_2^*))u_1^*$. Since $f\in X$ then \eqref{higherreg} follows.
\section{Proof of Theorem 2}
Using the theory of analytic semigroups we first establish the local well-posedness of $\eqref{System}$. Using quasipositivity of the right hand side of $\eqref{System}$ we next prove that the generated semiflow preserves nonnegativity of initial conditions. Then using a compensation effect  we derive $L_{\infty}(0,\infty,X)$ estimate for $u_3,u_4,u_5$ and $L_{\infty}(0,\infty,Y)$ estimate for $u_1,u_2$. Finally thanks to the regularising properties of the semigroup $e^{A_Yt}$ we bootstrap the estimate to $\bold{u}\in L_{\infty}(0,\infty,X^5)$.

\subsection{Local existence}

We rewrite system \eqref{System} in the new variables
$\bold{z}=\bold{u}-\bold{u}^*$, where $\bold{u}^*$ is the unique steady state of $\eqref{System}$, and put it into the semigroup framework:
\bs\label{abstract}
\eq{
\bold{z}'-\bold{A_X}\bold{z}&=\bold{f}(\bold{z}), && t>0\\
\bold{z}(0)&=\bold{z_0}=\bold{u_0}-\bold{u}^*,
}
\es
where 
\begin{align*}
\bold{A_X}&=A_X\times (dA_X)\times 0^3\\
\bold{f}&=(f_1,f_2,f_3,f_4,f_5):X^5\to X^5\\
f_1(\bold{z})&=-(b_1+c_1)z_1-(z_1z_3+u_1^*z_3+u_3^*z_1)+c_2z_2+c_4z_4\\
f_2(\bold{z})&=-(b_2+c_2)z_2-c_3(z_2z_3+u_3^*z_2+u_2^*z_3)+c_1z_1+c_5z_5\\
f_3(\bold{z})&=-b_3z_3-(z_1z_3+u_1^*z_3+u_3^*z_1)-c_3(z_2z_3+u_3^*z_2+u_2^*z_3)+c_4z_4+c_5z_5\\
f_4(\bold{z})&=-(b_4+c_4)z_4+(z_1z_3+u_1^*z_3+u_3^*z_1)\\
f_5(\bold{z})&=-(b_5+c_5)z_5+c_3(z_2z_3+u_3^*z_2+u_2^*z_3).
\end{align*}
Observe that $\bold{A_X}$ generates an analytic, strongly continuous semigroup in $X^5: \\  e^{t\bold{A_X}}=e^{tA_X}\times e^{tdA_X}\times (Id)^3$. Moreover $\bold{f}$ is Lipschitz continuous on bounded subsets of $X^5$. Using Lemma \eqref{lem3} we obtain that \eqref{abstract} possesses a unique solution defined on a maximal time interval $[0,T_{max})$ with the following regularity:
\begin{align*}
z_1,z_2&\in C([0,T_{max});X)\cap C^1((0,T_{max});X)\cap C((0,T_{max});X_1)\\
z_3,z_4,z_5&\in C^1([0,T_{max});X).
\end{align*}
Setting $\bold{u}=\bold{z}+\bold{u^*}$ it is obvious that $\bold{u}$ is the unique solution to \eqref{System}.
\subsection{Positive cone invariance}

Consider the following system
\bs\label{SystemV}
\eq{
\partial_t v_1-\partial^2_{xx} v_1&=-(b_1+c_1+(v_3)_+)v_1+c_2(v_2)_++c_4(v_4)_++p_1\delta,&&(t,x)\in I_{\infty}\label{SystemVA}\\
\partial_t v_2-d\partial^2_{xx} v_2&=-(b_2+c_2+c_{3}(v_3)_+)v_2+c_1(v_1)_++c_5(v_5)_+,&&(t,x)\in I_{\infty}\label{SystemVB}\\
\partial_t v_3&=-(b_3+(v_1)_++c_{3}(v_2)_+)v_3+c_4(v_4)_++c_5(v_5)_++p_3,&&(t,x)\in I_{\infty}\label{SystemVC}\\
\partial_t v_4&=-(b_4+c_4)v_4+(v_1)_+(v_3)_+,&&(t,x)\in I_{\infty}\label{SystemVD}\\
\partial_t v_5&=-(b_5+c_5)v_5+c_{3}(v_2)_+(v_3)_+,&&(t,x)\in I_{\infty}\label{SystemVE}
}
\es
with boundary and initial conditions
\begin{align*}
\partial_x v_1&=\partial_x v_2=0,&&(t,x)\in(\partial I)_{\infty}\\
\mathbf{v}(0,\cdot)&=\bold{u_0},&&x\in I
\end{align*}

Reasoning as in the previous section, the system \eqref{SystemV} possesses a unique maximally defined solution $\bold{v}(t)$ on $[0,T_{max}')$ in $X^5$. We will now prove that $\bold{v(t)}\geq0$ for $t\in[0,T_{max}')$. Multiplying \eqref{SystemV} by $-(\bold{v})_-$ and adding equations we obtain:
\begin{align*}
d/dt\sum_{i=1}^5\n{(v_i)_-}^2_2+\n{\partial_x(v_1)_-}_2^2+2d\n{\partial_x(v_2)_-}_2^2&\leq0\\
\sum_{i=1}^5\n{(v_i(0))_-}^2_2&=0.
\end{align*}
Thus $\bold{v}(t)\geq0$ for $t\in[0,T_{max}')$. Then $\bold{v}_+=\bold{v}$ and it readily follows from $\eqref{SystemV}$ that $\bold{v}$ solves $\eqref{System}$ on $[0,T_{max}')$. Consequently, $\bold{v}=\bold{u}$ on $[0,T_{max}')$ and $T_{max}'\leq T_{max}$. Finally observe that if $T_{max}'<\infty$ then, by \eqref{blow}, $\limsup_{t\to T_{max}'}\n{\bold{u}(t)}_{X^5}=\limsup_{t\to T_{max}'}\n{\bold{v}(t)}_{X^5}=\infty$ thus $T_{max}'=T_{max}$ and $\bold{u(t)}\geq0$ on $[0,T_{max})$.

\subsection{$u_3,u_4,u_5\in L_{\infty}(0,T_{max};X)$}
Adding equations \eqref{SystemC},\eqref{SystemD},\eqref{SystemE} and using nonnegativity of $\bold{u}$, we obtain for every $x\in I$
\begin{align*}
\partial_t \sum_{i=3}^5u_i+\underline{b}\sum_{i=3}^5u_i\leq p_3, \quad t\in[0,T_{max}).
\end{align*}
Thus 
\begin{align}\label{u3u4u5infty}
0\leq \sum_{i=3}^5u_i\leq e^{-\underline{b}t}\sum_{i=3}^5u_{i0}+p_3(1-e^{-\underline{b}t})/\underline{b}, \quad t\in[0,T_{max}).
\end{align}
\subsection{$T_{max}=\infty$}
Observe that, due to \eqref{u3u4u5infty}, $\bold{f(z(t))}$ satisfies \eqref{blowcond} with $V=X^5$ hence $T_{max}=\infty$ by Lemma \eqref{lem3} .

\subsection{$u_1,u_2\in L_{\infty}(0,\infty;Y)$}
After integrating equations \eqref{SystemA}, \eqref{SystemB}, \eqref{SystemD}, \eqref{SystemE} over the set $I$ and adding them together we obtain
\begin{align*}
\frac{d}{dt}(\sum_{i\in\{1,2,4,5\}}\n{u_i}_Y)+\underline{b}\sum_{i\in\{1,2,4,5\}}\n{u_i}_Y\leq p_1,
\end{align*}
Thus
\begin{align}\label{u1u2u4u5L1}
\sum_{i\in\{1,2,4,5\}}\n{u_i(t)}_Y&\leq e^{-\underline{b}t}\sum_{i\in\{1,2,4,5\}}\n{u_{i0}}_Y+p_1(1-e^{-\underline{b}t})/\underline{b}.
\end{align}

\subsection{$u_1,u_2\in L_{\infty}(0,\infty;X)$}
From \eqref{u3u4u5infty}, \eqref{u1u2u4u5L1} we obtain that $f_1(\mathbf{z})+z_1\in L_{\infty}(0,\infty;Y)$.
Using the Duhamel formula and estimates from Lemma \eqref{lem1} we get
\begin{align*}
\n{z_1(t)}_X&\leq e^{-t}\n{e^{tA_X}}_{\mathcal{L}(X)}\n{z_{10}}_X+\int_0^te^{-s}\n{e^{sA_Y}}_{\mathcal{L}(Y,X)}\n{f_1(\bold{z}(t-s))+z_1(t-s)}_{Y}ds\\
&\leq\n{z_{10}}_X+C\int_0^{\infty}(1\wedge s)^{-1/2}e^{-s}ds\n{f_1(\mathbf{z})+z_1}_{L_{\infty}(Y)}\leq\n{z_{10}}_X+C\n{f_1(\mathbf{z})+z_1}_{L_{\infty}(Y)},
\end{align*}
whence $u_1\in L_{\infty}(0,\infty;X)$. A similar argument gives $u_2\in L_{\infty}(0,\infty;X)$ and completes the proof.

\section{Acknowledgement}
The author would like to express his gratitude towards his PhD supervisors Philippe Lauren\c{c}ot and Dariusz Wrzosek for their constant encouragement and countless helpful remarks and towards his numerous colleagues for stimulating discussions. \\
The author was supported by the International Ph.D. Projects Programme of Foundation for Polish Science operated within the Innovative Economy Operational Programme 2007-2013 funded by European
Regional Development Fund (Ph.D. Programme: Mathematical Methods in Natural Sciences). \\
Part of this research was carried out during the author's visit to the Institut de Math\'ematiques de Toulouse, Universit\'e Paul Sabatier, Toulouse III.

\end{document}